\documentclass[a4paper,article]{amsproc}
\usepackage{amssymb}
\usepackage[hyphens]{url} 
\usepackage{amsmath,amsthm,amssymb}
 \usepackage{mathdots}
 \usepackage{graphicx}
 \usepackage{tikz}
 \usetikzlibrary{decorations}
 \usepackage{verbatim}
 \usepackage{booktabs}
  \usepackage{array}
  \usepackage{tabularx}
 \usepackage{booktabs}
\usepackage{siunitx}
\usepackage{rotating}
\usepackage{adjustbox}
\usepackage{longtable}
\usepackage{caption} 
\usepackage{tikz-cd}

\newcommand\Tstrut{\rule{0pt}{2.9ex}}         
\newcommand\Bstrut{\rule[-1.65ex]{0pt}{0pt}}   

\newcolumntype{C}[1]{>{\centering\arraybackslash}p{#1}}

\newtheorem{thm}{Theorem}[section]
\theoremstyle{theorem}
\newtheorem{conj}[thm]{Conjecture}
\newtheorem{claim}[thm]{Claim}
\newtheorem{cor}[thm]{Corollary}

\newtheorem{remark}[thm]{Remark}
\newtheorem{prop}[thm]{Proposition}

\DeclareMathOperator{\disc}{disc}

\DeclareMathOperator{\End}{End}

\DeclareMathOperator{\Pic}{Pic}

\newcommand{\F}{\mathbb{F}}

\newcommand{\N}{\mathbb{N}}
\newcommand{\Z}{\mathbb{Z}}
\newcommand{\Q}{\mathbb{Q}}

\newcommand{\C}{\mathbb{C}}

\newcommand{\Pgotic}{\mathfrak{P}}
\newcommand{\pgotic}{\mathfrak{p}}

\newcommand{\Ocal}{\mathcal{O}}
\newcommand{\agotic}{\mathfrak{a}}

\newcommand{\Ogotic}{\mathcal{O}}

\newcommand{\Qbar}{\overline{\Q}}

\newcommand{\Agotic}{\mathfrak{A}}

\renewcommand{\mod}{\operatorname{mod}\ }
\newcommand{\Mod}[1]{\ \mathrm{mod}\ #1}

\usepackage{nicefrac}

\renewcommand{\leq}{\leqslant}
\renewcommand{\geq}{\geqslant}

\title[On singular moduli that are $S$-units]{On singular moduli that are $S$-units}



\author[Francesco Campagna]{ Francesco Campagna \\ \\ \small{University of Copenhagen, Department of Mathematics, Universitetsparken 5} \\
\small{campagna@math.ku.dk}, ORCID: 0000-0001-6249-1915.} 

\address{University of Copenhagen, Department of Mathematics, Universitetsparken 5, Copenhagen}
\email{campagna@math.ku.dk}

\begin{document}

\begin{abstract}
Recently Yu. Bilu, P. Habegger and L. K\"uhne proved that no singular modulus can be a unit in the ring of algebraic integers. In this paper we study for which sets $S$ of prime numbers there is no singular modulus that is an $S$-unit. Here we prove that if $S$ is the set of all primes $p$ congruent to 1 modulo 3, no singular modulus is an $S$-unit. We then give some remarks on the general case and we study the norm factorizations of a special family of singular moduli. 
\end{abstract}

\maketitle

\section{INTRODUCTION}  

Elliptic curves with complex multiplication and their $j$-invariants, the so-called \textit{singular moduli}, have been studied for a long time in number theory because of their many interesting arithmetic properties. For instance, singular moduli are always algebraic integers and they generate ring class fields relative to orders in imaginary quadratic fields (\cite{Cox}). In 1985, Gross and Zagier proved, under certain technical assumptions, a formula concerning the prime factorization of the difference of two singular moduli (\cite{Gross}). This formula has been recently generalized by Lauter and Viray in \cite{Viray}. 

A few years ago Masser, motivated by some effective results of Andr\'e-Oort type (see for instance \cite{Zannier} and \cite{Kuhne}), raised the question whether it is possible that a singular modulus is a unit in the ring of algebraic integers. Habegger has shown in \cite{Habbeger} that there exists at most a finite number of singular moduli that are algebraic units. However the argument used in the proof is non-effective since it relies on an equidistribution result. In a more recent work of Bilu, Habegger and K\"uhne this question has been completely settled with the proof of the following
\begin{thm}[Theorem 1.1 in \cite{Bilu}]  \label{Habbeger}
There are no singular moduli that are units in the ring of algebraic integers.
\end{thm}

The theorem above opened the way to a number of interesting questions. For instance one may ask, inspired by the work of Gross and Zagier, whether there are pairs of singular moduli whose difference is a unit. Work in this direction appears in \cite{Yingkun}, where Li proves that the difference of two singular moduli is never a unit. 

Another possible research path is the following: let $S$ be a finite set of prime numbers. What is the number of singular moduli that are $S$-units? By the theorem above, if $S= \emptyset$ the answer is zero. However if $S \neq \emptyset$ there can certainly exist singular moduli that are $S$-units. For instance, for $S=\{2,3\}$, the integers $12^3, -32^3$ and $-96^3$ are three singular moduli that are $\{2,3\}$-units (they are the $j$-invariants of elliptic curves having complex multiplication by $\Z[\sqrt{-1}], \Z[\frac{1+\sqrt{-11}}{2}]$ and $\Z[\frac{1+\sqrt{-19}}{2}]$ respectively). We are then looking for some finiteness statement that will in general depend on the choice of the set $S$. Recently Herrero, Menares and Rivera-Letelier announced a proof of the finiteness of the set of singular moduli that are $S$-units for any finite set of primes $S$. However, to the best of our knowledge, their argument is not effective. In this paper we give the first effective results in this direction; in particular our main theorem is the following

\begin{thm} \label{main theorem}
Let $S$ be the set of rational primes congruent to $1$ modulo $3$. If a singular modulus is an $S$-unit, then it is a unit.
\end{thm}

Combining Theorem \ref{main theorem} with Theorem \ref{Habbeger} above we immediately get

\begin{cor}\label{main corollary}
Let $S$ be the set of rational primes congruent to 1 modulo 3. Then no singular modulus is an $S$-unit.
\end{cor}

Notice that the set of primes considered in Corollary \ref{main corollary} is infinite: hence the corollary gives an effective answer to the problem of singular $S$-units for infinitely many finite sets $S$ of rational primes.

As we will see, the proof of Theorem \ref{main theorem} naturally leads to the study of $j$-invariants of elliptic curves that have complex multiplication by orders in $\Q(\sqrt{-3})$. Hence in the final part of this paper some properties of this family of singular moduli are pointed out and studied. 

We develop our article as follows. In section \ref{preliminaries} we fix the notation and we recall some theorems from the theory of complex multiplication. In section \ref{Proof} we give the proof of Theorem \ref{main theorem}. The main ingredient of the proof will be Deuring's reduction theory for CM elliptic curves. Conversely, in section \ref{Remarks} we prove that for every finite set $S$ of prime numbers there is always a singular modulus that is not an $S$-unit (Proposition \ref{Proposition}). In section \ref{finale}, motivated also by numerical computations, we discuss some properties of the singular $j$-invariants relative to orders in $\Q(\sqrt{-3})$. Here we will make heavy use of the formulas proved in \cite{Viray}. As an appendix we include a table containing some explicit factorizations for the norms of these $j$-invariants. We dedicate a small final section to the study of the case $j-1728$, kindly suggested by Habegger.

\section{PRELIMINARIES AND NOTATION} \label{preliminaries}

The main references for this section are \cite{Cox}, \cite{Lang}, \cite{Silverman} and \cite{Silverman2}. We begin by recalling the notion of $S$-units. Let $K$ be a number field and $S$ a set of rational primes. We say that an element $x\in K$ is an \textit{$S$-unit} if, for every prime $p \not \in S$ and for every prime $\pgotic$ of $K$ lying over $p$, the element $x$ is a unit in the ring of integers of $K_{\pgotic}$ (here $K_{\pgotic}$ denotes the completion of $K$ at the prime $\pgotic$). In other words $x$ is an $S$-unit if and only if $x\neq 0$ and all the primes appearing in the prime ideal factorization of $x\Ogotic_K$ lie above primes in $S$. If $x$ is an algebraic integer, this is equivalent to require that all the primes dividing its absolute norm over $\Q$ are in $S$. 

The principal object of study of this paper are singular moduli. A singular modulus is the $j$-invariant $j(E)$ of an elliptic curve $E$ defined over $\C$ with complex multiplication. For a positive integer $D$ congruent to $0$ or $3$ modulo $4$, we say that a singular modulus is of discriminant $-D$ if it is the $j$-invariant of an elliptic curve with complex multiplication by an order of discriminant $-D$. Recall that, given an imaginary quadratic order $\mathcal{O}$, there is a bijection between the class group $\operatorname{Cl}(\mathcal{O})$ and the set of elliptic curves $E/\C$ with complex multiplication by $\mathcal{O}$ up to $\C$-isomorphism. For a proper fractional ideal $\mathfrak{a}$ of $\mathcal{O}$ we then denote by $j(\mathfrak{a})$ the $j$-invariant of any elliptic curve $E/\C$ whose $\C$-isomorphism class corresponds to the class $[\mathfrak{a}] \in \operatorname{Cl}(\mathcal{O})$ under this bijection.

As we already mentioned in the introduction, singular moduli are algebraic integers. Given a singular modulus of discriminant $-D$ we will denote by $H_D(x) \in \Z[x]$ its minimal polynomial over $\Q$ and we call it the \textit{Hilbert class polynomial of discriminant} $-D$. It is well-known that the roots of the Hilbert class polynomial of discriminant $-D$ are all the singular moduli of discriminant $-D$. In particular $\deg H_D(x)=h_D$, where $h_D$ is the class number of the unique order of discriminant $-D$, and all the singular moduli of discriminant $-D$ have the same absolute norm over the rationals. Indeed, if $N_{L/\Q}(\cdot )$ denotes the usual norm map from a number field $L$ to $\Q$, we see that  $|N_{\Q(j)/\Q}(j) |=|H_D(0)|$  for every $j$ of discriminant $-D$. We finally recall that a discriminant $-D$ is called a \textit{fundamental discriminant} if it is the discriminant of a maximal order in an imaginary quadratic field. If $j$ is a singular modulus relative to a fundamental discriminant $-D$ then $\Q(\sqrt{-D}, j)$ is the Hilbert class field of $\Q(\sqrt{-D})$. 

In the sequel we will need some theorems from the reduction theory of elliptic curves with complex multiplication. We recall here the main results that will be used in the article. If $E/L$ is an elliptic curve with complex multiplication defined over a number field $L$, then $E$ has potential good reduction at every prime of $L$, see \cite[II, Theorem~6.4]{Silverman2}. This means that for every prime $\pgotic$ of $L$ there exists a finite extension $F/L$ such that $E$ has good reduction at every prime of $F$ lying over $\pgotic$. The reduction type of CM elliptic curves was first described by Deuring in \cite{Deuring} and it is nowadays known as "Deuring theory". We recall some important results from that theory, starting with the following theorem (compare with Theorem 13.12 in \cite{Lang}).

\begin{thm}\label{Lang}
Let $E$ be an elliptic curve defined over a number field $L$ and with complex multiplication by an order $\Ocal=\Z+f\Ocal_K$ of conductor $f$ in an imaginary quadratic field $K$. Let $\Pgotic$ be a prime of $L$ lying over a rational prime $p$ where $E$ has good reduction $\widetilde{E}$. Then the reduction $\text{mod } \Pgotic$ induces an injective homomorphism $\End_{L}(E) \hookrightarrow \End_{\overline{\F}_p}(\widetilde{E})$. Moreover:
\begin{enumerate}
\item if $p$ is ramified or inert in $K/\Q$, then $\End_{\overline{\F}_p}(\widetilde{E})$ is isomorphic to an order in a quaternion algebra;
\item if $p$ is split in $K/\Q$, then $\End_{\overline{\F}_p}(\widetilde{E})\cong \Z+f'\Ocal_K$ where $f=f'p^n$ with $\gcd(f',p)=1$;
\end{enumerate}
\end{thm}
 For every singular modulus $j$ denote by $\Ocal_j$ and $K_j$ respectively the endomorphism order and the endomorphism algebra of an elliptic curve $E$ with invariant $j$. For each rational prime $p$ let $\mathcal{J}_p$ be the set of singular moduli $j$ such that $p$ splits in $K_j$ and $p$ does not divide the conductor of the order $\Ocal_j$. By choosing a prime $\Pgotic$ of $\Qbar$ above $p$ (\textit{i.e.} a compatible system of finite primes lying above $p$) we obtain a reduction map
$$\mathcal{J}_p \to \overline{\F}_p.$$
By Theorem \ref{Lang} the image of this map lies in the set of non-supersingular invariants modulo $p$. However, one can prove more, as the following theorem shows (see Theorem 13.13 in \cite{Lang}).

\begin{thm} \label{Lang2}
The map $\mathcal{J}_p \to \overline{\F}_p$ is a bijection of $\mathcal{J}_p$ with the set of non-supersingular invariants modulo $p$.
\end{thm}

\begin{cor} \label{Cox}
Let $\Ocal_1$ and $\Ocal_2$ be orders in imaginary quadratic fields $K_1$ and $K_2$ of conductors $f_1, f_2$ respectively. For $i=1,2$ let $\agotic_i$ be a proper fractional $\Ocal_i$-ideal such that $j(\agotic_1) \neq j(\agotic_2)$. Suppose that $L$ is a number field containing $j(\agotic_1)$ and $j(\agotic_2)$, and let $\Pgotic$ be a prime of $L$ lying over a rational prime $p$. If $j(\agotic_1) \equiv j(\agotic_2) \Mod \Pgotic$ then either $K_1=K_2$ and $p$ divides $f_1f_2$ or $p$ is non-split in $K_1$ and $K_2$.
\end{cor}

\begin{proof}
For $i=1,2$ let $E_i$ be an elliptic curve defined over $L$ with complex multiplication by $\Ocal_i$ and $j$-invariant $j(\agotic_i)$. After base-changing to a finite field extension, we can assume without loss of generality that the elliptic curves $E_i$ have good reduction at $\Pgotic$ and that all their endomorhisms are defined over the base field (see \cite[II, Theorem~2.2~(b)]{Silverman2}).
The hypothesis $j(\agotic_1) \equiv j(\agotic_2) \mod \Pgotic$ then implies that the reduced elliptic curves $\widetilde{E_1}$ and $\widetilde{E_2}$ are isomorphic over $\overline{\F}_p$. In particular their endomorphism rings over $\overline{\F}_p$ must be isomorphic. By Theorem \ref{Lang} these rings are either both isomorphic to an order in a quaternion algebra or to an order in an imaginary quadratic field. In the first case, $p$ is non-split in $K_1$ and in $K_2$. In the second case, $p$ splits in both $K_1$ and $K_2$, and we have
$$\End_{\overline{\F}_p}(\widetilde{E_i}) \cong \Z+f_i'\Ocal_{K_i} \hspace{1cm} \text{for} \ i=1,2$$
where $\Ocal_{K_i}$ is the ring of integers of the field $K_i$ and $f_i=f_i'p^{n_i}$, $\gcd(f_i',p)=1$. Since the two endomorphism rings must be isomorphic, we must have $K_1=K_2$ and $f_1'=f_2'$. If $p$ divides one among $f_1$ and $f_2$ we are done. Suppose on the contrary that $p\nmid f_1f_2$; then, by the discussion above, we have $f_1=f_2$. But then $j(\agotic_1)$ and $j(\agotic_2)$ cannot be congruent modulo $\Pgotic$ by Theorem \ref{Lang2}. This contradiction concludes the proof.
\end{proof}

The results above link the splitting behaviour of primes in imaginary quadratic fields to the reduction type of elliptic curves with complex multiplication. They will be crucial in the proof of our main result.

\section{PROOF OF THEOREM \ref{main theorem}} \label{Proof}
Throughout this section, for a given set of rational primes $S$, a singular modulus which is an $S$-unit will be called a singular $S$-unit.
\begin{proof}[Proof of Theorem \ref{main theorem}]
Let $S$ be the set of rational primes congruent to $1$ mod $3$. First we show that every singular $S$-unit corresponding to a quadratic field $K\neq \Q(\sqrt{-3})$ is in fact a unit. This follows almost immediately from Corollary $\ref{Cox}$.

 Let $j$ be a singular modulus corresponding to  a field $K\neq \Q(\sqrt{-3})$, let $L=\Q(j)$ and consider a prime $\Pgotic$ of $L$ lying over a rational prime $p$ congruent to $1$ mod $3$. Notice that $j_0=0$ is the unique singular modulus corresponding to the maximal order of the quadratic field $\Q(\sqrt{-3})$ and that $p$ splits in the latter. Then by Corollary \ref{Cox} it follows that the prime $\Pgotic$ cannot divide $j$. Otherwise we would have $j \equiv j_0 \mod \Pgotic$, which contradicts the fact that $p$ splits in $\Q(\sqrt{-3})$. In particular every singular $S$-unit associated with an imaginary quadratic field of discriminant less than $-3$ is in fact a unit. 
 
We are left to study all the singular $S$-units relative to the field $\Q(\sqrt{-3})$. For these we are going to prove the following 
\begin{claim} \label{claim}
Let $j$ be a singular modulus associated with an order $\Ocal_j$ contained in $\Q(\sqrt{-3})$. Then the primes $2,3$ and $5$ divide $N_{\Q(j)/\Q}(j)$, where $N_{\Q(j)/\Q}(\cdot)$ denotes the usual norm function on number fields.
\end{claim}
It is clear that Theorem \ref{main theorem} follows immediately from the claim. Indeed a prime $p$ divides $N_{\Q(j)/\Q}(j)$ if and only if there exists a prime $\Pgotic$ in $\Q(j)$ lying over $p$ such that $\Pgotic$ divides $j$. 

We prove Claim \ref{claim} only for the prime $p=3$, the other cases being analogous. Fix $j$ as in the statement of the claim and let $E/L$ be an elliptic curve defined over a number field $L$ with complex multiplication by the order $\Ocal_j$. Since $E$ has potential good reduction at every prime of $L$ (Theorem 6.4, Chapter II in \cite{Silverman2}) we may assume that $E$ has good reduction at every prime $\Pgotic$ of $L$ lying over 3. Fix such a prime. Since 3 ramifies in $\Q(\sqrt{-3})$, the reduced elliptic curve $\widetilde{E}= E$ mod $\Pgotic$ is supersingular by Theorem \ref{Lang}. However by Theorem V, 4.1 (c) in \cite{Silverman} there is only one isomorphism class of supersingular elliptic curves over $\overline{\F}_3$, one representative being given by
$$E_0: y^2=x^3+x.$$
Now we have $j(E_0)=1728$ which is divisible by 3 and then
$$\widetilde{j(E)}=j(\widetilde{E})=j(E_0)=0$$
where $\sim$ denotes the reduction modulo $\Pgotic$. We deduce that $\Pgotic$ divides $j$ and by the discussion above this concludes the proof. 

If $p=2$ or $p=5$ the argument can be repeated; in these cases the only supersingular elliptic curves over $\overline{\F}_p$ are given by
$$E_1:y^2+y=x^3$$
if $p=2$ and by
$$E_2: y^2=x^3+1$$
if $p=5$. In both cases the $j$-invariant is zero. This concludes the proof of Theorem \ref{main theorem}.
\end{proof}

In Table \ref{Tabella} we have collected the factorizations of $N_{\Q(j)/\Q}(j)$ for all the singular moduli $j$ of discriminant $-3f^2$ with $f=1,...,50$. The table illustrates, among other phenomena to be discussed later, the statement of Claim \ref{claim}. 

\section{SOME REMARKS ON SINGULAR S-UNITS} \label{Remarks}

The result given by Corollary \ref{main corollary} is a special effective case of Theorem 1.1 for $S$-units. For $S$ taken to be the infinite set of primes congruent to $1$ mod $3$, the emptiness of the set of singular $S$-units is proved. On the other hand, for different collections of primes $S$ the set of singular $S$-units in general can be non-empty. For instance we already showed in the introduction that this is the case with the set $S=\{2,3\}$. However, since in general singular moduli are  ``highly factorable" algebraic integers, one would expect that in general the set of singular $S$-units is finite \textit{for every finite set of primes} $S$. \par
As we mentioned in the introduction Herrero, Menares and Rivera-Letelier announced a proof of the fact that, for every finite set of primes $S$, there are at most finitely many singular $S$-units. The first part of their work \cite{HMR} (the only one accessible at the time of writing) does not contain the details of the proof. We give here a much weaker statement than the one mentioned above, which however provides evidence for the stronger claim. Proposition \ref{Proposition} can be considered as a weak converse to Theorem \ref{Habbeger}. The main ingredient of the proof, besides Theorem 1.1, is the reduction theory for CM elliptic curves as described in Section 2.

\begin{prop} \label{Proposition}
Let $S=\{p_1,...,p_n\}$ be a finite set of rational primes. Then there exist infinitely many singular moduli of fundamental discriminant that are not $S$-units. 
\end{prop}

\begin{proof}
We know by Theorem \ref{Habbeger} that no singular modulus is a unit. In particular there is always a rational prime that divides the norm of any singular modulus. If $j$ is a singular modulus of fundamental discriminant $-D<-3$ and $p$ divides $N_{\Q(j)/\Q}(j)$, then by Corollary \ref{Cox} the prime $p$ cannot split in $\Q(\sqrt{-D})$. The idea for the proof of the proposition is then to find infinitely many fundamental discriminants $-D$ such that all the primes in $S$ split in $\Q(\sqrt{-D})$. In this way the set of primes dividing the norm of any singular modulus of discriminant $-D$ (this set is nonempty by the above discussion) has trivial intersection with $S$. \\
Let $q$ be a prime number such that
\begin{itemize}
\item $q \equiv -1$ mod $p_i$ for every $p_i \in S$.
\item $q \equiv -1$ mod $8$.
\end{itemize}
We know that there are infinitely many primes satisfying these conditions by Dirichlet's theorem on primes in arithmetic progression and the Chinese reminder theorem. We claim that in $\Q(\sqrt{-q})$ all the primes of $S$ are split. First of all notice that $\disc \Q(\sqrt{-q})=-q$ because clearly $q$ is squarefree and $-q \equiv 1 \mod 4$ by assumption. To prove that every prime in $S$ splits in this field we compute the Kronecker symbols $(-q/p_i)$. We have two cases: 
\begin{itemize}
\item If $p_i=2$ for some $i$ then $\left(\frac{-q}{2} \right)=1$ because $-q \equiv 1 \mod 8$.
\item If $p_i >2$ then
$$\left(\frac{-q}{p_i} \right)=\left(\frac{-1}{p_i} \right) \left(\frac{q}{p_i} \right)= (-1)^{\frac{p_i -1}{2}} \cdot \left(\frac{q}{p_i} \right)=(-1)^{\frac{p_i -1}{2}} \cdot \left(\frac{-1}{p_i} \right)= 1 $$
because $(-1/p_i)=(-1)^{\frac{p_i -1}{2}}$ for every odd prime $p_i$.
\end{itemize}
Since the Kronecker symbols above are equal to 1, we deduce that all the primes in $S$ are split in $\Q(\sqrt{-q})$. This proves the proposition.
\end{proof}

\begin{remark}
Using the same strategy, one can actually prove that for every finite set $S$ of rational primes there exist a positive proportion of negative fundamental discriminants whose corresponding singular moduli are not $S$-units. Here is a sketch of the argument: as explained in the proof of Proposition \ref{Proposition}, it suffices to consider the set of fundamental discriminants $-D$ for which every prime in $S$ is split in $\Q(\sqrt{-D})$. The map $d \mapsto \Q(\sqrt{-d})$ gives a bijection between the set of squarefree positive integers and the set of imaginary quadratic fields. Under this bijection, the imaginary quadratic fields where every prime in $S$ splits correspond to the squarefree integers $d$ such that the equality of Kronecker symbols $(d/p)=(-1/p)=(-1)^{\frac{p-1}{2}}$ holds for all odd $p \in S$ and satisfying $d\equiv 7 \text{ mod } 8$ if $2 \in S$. We are then reduced to find the proportion $\delta_S$ of all the positive squarefree numbers $d$ satisfying these congruence conditions. By the Chinese Reminder Theorem, this is equivalent to studying the asymptotic distribution of squarefree numbers in the residue classes mod $N:=4 \cdot \prod_{p\in S} p$. 

For $a,k \in \N$ we denote by $Q(x;a,k)$ the number of squarefree positive integers $d \leq x$ such that $d\equiv a \text{ mod } k$. The study of the asymptotic behaviour of the function $Q(x;a,k)$ dates back to Landau \cite{Landau} pp. 633-636. An equivalent formulation of his results is given in  \cite[Lemma 8]{Schwartz}, which in particular yields
\begin{equation*} 
Q(x;a,k) \thicksim \frac{6}{\pi^2} x \cdot \delta(a,k) \hspace{0.5cm} \text{ as } x \rightarrow \infty,
\end{equation*}
where
\begin{equation} \label{squarefree}
\delta(a,k):=\frac{1}{k} \prod_{p \mid k} \frac{1}{1-p^{-2}} \cdot \prod_{\substack{p\mid (a,k), \\ (p^2,k) \mid a}} \left(1-\frac{(p^2,k)}{p^2} \right).
\end{equation}

In the above formula $p$ always denotes a prime number and for every pair of integers $u,v \in \Z$ we have set $(u,v):=\gcd(u,v)$. Since the natural density of the set of squarefree positive integers is $6/\pi^2$ $($see \cite[Theorem 2.2 pag. 36]{Montgomery-Vaughan}$)$ the number $\delta(a,k)$ in \eqref{squarefree} represents the proportion of squarefree positive integers $d\equiv a \text{ mod } k$. Notice that for every $a \in \N$ the arithmetic function $\delta(a,\cdot)$ is multiplicative. Moreover, one has
\begin{align*}
\delta(a,p)=\frac{p}{p^2-1} \hspace{0.3cm} &\text{for all } p\neq 2 \text{ and } p\nmid a, \\
\delta(a,8)=\frac{1}{6} \hspace{0.3cm} &\text{for all } a \not \equiv 0,4 \text{ mod } 8
\end{align*}
and from this it is easy to deduce that $\delta_S=\prod_{p\in S} \delta_{S,p}$ with
\begin{equation*}
\delta_{S,p}=\frac{p}{2p+2} \hspace{0.2cm}\text{ if } p\neq 2, \hspace{0.5cm} \delta_{S,2}=\frac{1}{6}.
\end{equation*}
In particular $\delta_S >0$, as we wanted to show.
\end{remark}

\begin{remark}
If $S$ does not contain the primes $2,3$ or $5$, Proposition \ref{Proposition} could be proved by considering imaginary quadratic fields where these primes do not split. For instance, if $2 \not \in S$ we may consider the fundamental discriminants $-D \equiv 0 \mod 4$: then $2$ ramifies inside $\Q(\sqrt{-D})$ and by Theorem \ref{Lang} it is a prime of supersingular reduction for any elliptic curve with complex multiplication by the ring of integers in $\Q(\sqrt{-D})$. Since $0$ is the only supersingular $j$-invariant modulo 2, we deduce that singular moduli of discriminant $D$ cannot be $S$-units.
\end{remark}

One of the most remarkable aspects of Corollary \ref{main corollary} is that it proves the emptiness of the set of singular $S$-units for an \textit{infinite} set of primes $S$. A natural question that arises is whether there exist other infinite collections of primes $S$ for which one could prove at least the finiteness of the set of singular $S$-units. The next proposition gives a negative answer to this question when almost all the primes congruent to $2$ modulo $3$ are contained in $S$.

\begin{prop}
Let $S$ be a set of rational primes containing almost all primes congruent to $2$ mod $3$ (that is, containing all but finitely many primes congruent to $2$ mod $3$). Then there are infinitely many singular $S$-units.
\end{prop}

\begin{proof}
Let $\{l_1,\dots,l_n\}$ be the finite set of all primes $l_i \equiv 2$ mod $3$ not contained in $S$. As in the proof of Proposition \ref{Proposition} there are infinitely many negative fundamental discriminants $-D$ such that for every $i=1,...,n$ the prime $l_i$ splits in $\Q(\sqrt{-D})$. Then the singular moduli corresponding to these discriminants cannot be divided by any prime $p \not \in S$. This follows from Corollary \ref{Cox} and the fact that we are considering fundamental discriminants. Hence there are infinitely many singular $S$-units.
\end{proof}

\section{SINGULAR MODULI OF DISCRIMINANT $-3f^2$} \label{finale}

We have seen that the proof of Theorem \ref{main theorem} naturally leads to the study of singular moduli relative to orders $\Ogotic_D \subseteq \Q(\sqrt{-3})$ \textit{i.e}. of singular moduli of discriminant $-3f^2$, $f \in \N_{>0}$ being the conductor of the corresponding order. We have collected some factorizations for the absolute value of the norm of these singular moduli in Table \ref{Tabella}. By looking at the table one immediately notices a main difference between the singular moduli in this family and all other $j$-invariants. Indeed we know (and we showed in the proof of Theorem \ref{main theorem}) that for a singular modulus $j$ relative to an order $\Ogotic \not \subseteq \Q(\sqrt{-3})$ the primes dividing $N_{\Q(j)/\Q}(j)$ cannot be congruent to 1 modulo 3. The situation is different for singular moduli of discriminant $-3f^2$. For instance we have:
\begin{align*}
&|N_{\Q(j)/\Q}(j_{21})|= 2^{30} \cdot 3^9 \cdot 5^6 \cdot 7 \cdot 17^3 \\
&|N_{\Q(j)/\Q}(j_{39})|=2^{66} \cdot 3^{21} \cdot 5^{12} \cdot 11^6 \cdot 13 \cdot 23^3 \\
&|N_{\Q(j)/\Q}(j_{57})|=2^{93} \cdot 3^{27} \cdot 5^{18} \cdot 11^6 \cdot 19 \cdot 29^6 \cdot 41^3 \cdot 53^3
\end{align*}
where $j_D$ is any singular modulus of discriminant $-D$. We see that $7,13$ and $19$ are all primes congruent to $1$ modulo $3$ and they appear in the above factorizations. A closer inspection of the table suggests more: whenever the conductor $f=p^n$ is an odd prime power, the prime $p$ divides the norm of the corresponding $j$ invariants with order exactly 1. For instance for $f=3,9,27,81$ we have:
\begin{align*}
f=3, \ &&D=-3\cdot 3^2 \hspace{0.5cm} &|N_{\Q(j)/\Q}(j)|=2^{15} \cdot 3  \cdot 5^3 \\
f=9, \ &&D=-3\cdot 3^4 \hspace{0.5cm} &|N_{\Q(j)/\Q}(j)|= 2^{45}\cdot 3 \cdot 5^9 \cdot 11^3 \cdot 23^3 \\
f=27, \ &&D=-3\cdot 3^6 \hspace{0.5cm} &|N_{\Q(j)/\Q}(j)|= 2^{144} \cdot 3\cdot 5^{27} \cdot 11^{15} \cdot 17^9 \cdot 23^3 \cdot 29^6 \cdot 53^6 \\
f=81, \ &&D=-3\cdot 3^8 \hspace{0.5cm} &|N_{\Q(j)/\Q}(j)|= 2^{432} \cdot 3 \cdot 5^{81} \cdot 11^{30} \cdot 17^{27} \cdot 23^6 \cdot 29^{12} \cdot 41^6 \cdot 47^9 \cdot  53^{12} \\
 & & & 
\hspace{2cm} \cdot 59^6 \cdot 71^6 \cdot 131^6 \cdot 167^6 \cdot 179^3 \cdot 191^6 \cdot 227^3 \cdot  239^3
\end{align*}
while for $f=5, 25, 125$ we have 
\begin{align*}
f=5, \ &&D=-3\cdot 5^2 \hspace{0.5cm} &|N_{\Q(j)/\Q}(j)|=2^{30} \cdot 3^6 \cdot 5 \cdot 11^3 \\
f=25, \ &&D=-3\cdot 5^4 \hspace{0.5cm} &|N_{\Q(j)/\Q}(j)|= 2^{156} \cdot 3^{48} \cdot 5 \cdot 11^9 \cdot 17^6 \cdot 23^6 \cdot 47^6 \cdot 59^3 \cdot 71^3 \\
f=125, \ &&D=-3\cdot 5^6 \hspace{0.5cm} &|N_{\Q(j)/\Q}(j)|=2^{810} \cdot 3^{150} \cdot 5 \cdot 11^{54} \cdot 17^{48} \cdot 23^{24} \cdot 29^{30} \cdot 41^{18} \cdot 53^{12}   \\
& & & \hspace{2cm} \cdot 59^{12} \cdot 71^{18} \cdot 83^{12} \cdot 89^{12} \cdot 107^6 \cdot 113^6 \cdot 131^6 \cdot 167^6 \\
& & &  \hspace{2cm} \cdot 179^9 \cdot 227^6 \cdot 251^6 \cdot 263^6 \cdot 311^3 \cdot 347^6 \cdot 359^3 
\end{align*}
We want to remark that the results displayed above are peculiar of non-maximal orders and can never be spotted for singular moduli of fundamental discriminants. Indeed we have the following

\begin{prop}
Let $j$ be a singular modulus of fundamental discriminant $-D$ and let $p$ be either $2,3$ or $5$. If $p$ divides $N_{\Q(j)/\Q}(j)$ then $p^2$ also divides $N_{\Q(j)/\Q}(j)$.
\end{prop}

\begin{proof}
One easily verifies the statement (for instance using SAGE \cite{Sage}) for all the fundamental discriminants $-D$ of class number $h_D\leq 2$. Hence we may suppose that $h_D>2$. 

Since $p$ divides $N_{\Q(j)/\Q}(j)$ there must exist a prime ideal $\pgotic$ of $\Q(j)$  lying over $p$ such that $\pgotic \mid j$. In particular $j \equiv 0 \mod \pgotic$ and since 0 is a singular modulus we can apply Theorem \ref{Cox} to deduce that $p$ does not split in $\Q(\sqrt{-D})$. Let now $E/L$ be an elliptic curve with complex multiplication whose $j$-invariant is a singular modulus of discriminant $-D$. Here we can assume that $L$ is a number field where $E$ has good reduction at all primes lying over $p$. Fix such a prime $\Pgotic$: since $p$ does not split in $\Q(\sqrt{-D})$, the reduction $\tilde{E}= E \mod \Pgotic$ is a supersingular elliptic curve by Theorem \ref{Lang}. But now we know by \cite[Chapter V, Theorem 4.1]{Silverman} that the only supersingular $j$-invariant modulo $p$ is 0. We deduce that all the singular moduli of discriminant $-D$ must reduce to 0 modulo $p$. In other words we have
$$H_D(x) \equiv x^{h_D} \mod p$$
where $H_D(x)$ is the Hilbert class polynomial of discriminant $-D$ and $h_D$ is the class number of the unique order of discriminant $-D$. 

Suppose now by contradiction that $p^2$ does not divide $N_{\Q(j)/\Q}(j)$. Then, by what we showed above, the Hilbert class polynomial $H_D(x)$ must be Eisenstein at $p$. Since $H_D(x)$ is the minimal polynomial of $j$, we deduce that $p$ has to be totally ramified in $\Q(j)$. Now we look at the field $\Q(\sqrt{-D}, j)$: since $[\Q(\sqrt{-D}, j):\Q(\sqrt{-D})]=[\Q(j):\Q]=h_D \geq 3$ there exists a prime of $\Q(\sqrt{-D})$ lying over $p$ that ramifies in $\Q(\sqrt{-D}, j)$. This contradicts the fact that $\Q(\sqrt{-D}, j)$ is the Hilbert class field of $\Q(\sqrt{-D})$. Hence $p^2$ divides $N_{\Q(j)/\Q}(j)$ and the proposition is proved.
\end{proof}

\begin{remark}
The same argument works if we consider discriminants of orders whose conductor is not divided by $2,3$ and $5$. Indeed in this case the associated ring class field is unramified at these primes.
\end{remark}

If the conductor $f$ is a power of $2$ the regularity in the factorizations appears a bit different from the previous cases. For instance for $f=2,4,8,16,32$ we have 
\begin{align*}
&f=2, \ &&D=-3\cdot 2^2 \hspace{0.5cm} &|N_{\Q(j)/\Q}(j)|=& 2^4 \cdot 3^3 \cdot 5^3 \\
&f=4, \ &&D=-3\cdot 2^4 \hspace{0.5cm} &|N_{\Q(j)/\Q}(j)|= &2^4 \cdot 3^9 \cdot 5^6 \cdot 11^3 \\
&f=8, \ &&D=-3\cdot 2^6 \hspace{0.5cm} &|N_{\Q(j)/\Q}(j)|=& 2^4 \cdot 3^{12} \cdot 5^{12} \cdot 11^6 \cdot 17^6 \cdot 23^3 \\
&f=16, \ &&D=-3\cdot 2^8 \hspace{0.5cm} &|N_{\Q(j)/\Q}(j)|=& 2^4 \cdot 3^{42} \cdot 5^{24} \cdot 11^6 \cdot 17^6 \cdot 23^3 \cdot 29^6 \cdot 41^6 \cdot 47^3 \\
&f=32, \ &&D=-3\cdot 2^{10} \hspace{0.5cm} &|N_{\Q(j)/\Q}(j)|=& 2^4 \cdot 3^{48} \cdot 5^{48} \cdot 11^{24} \cdot 17^6 \cdot 23^{12} \cdot 29^{12} \cdot 47^9 \cdot 53^6 \cdot \\
 & & & & & 59^6 \cdot 71^3 \cdot 83^6 \cdot 89^6
\end{align*}

All the data above suggest the following

\begin{conj} \label{conjecture}
For a prime number $p$ let $v_p: \Q^* \to \Z$ be the usual $p$-adic valuation and for a discriminant $D=-3f^2$ let $j_D$ be any singular modulus relative to that discriminant. Then:
\begin{itemize}
\item if $f=p^n$ with $p$ odd prime, $v_p(N_{\Q(j)/\Q}(j_D))=1$.
\item if $f=2^n$, $v_2(N_{\Q(j)/\Q}(j_D))=4$.
\end{itemize}
\end{conj}

We are able to prove part of Conjecture \ref{conjecture} in the following

\begin{thm}\label{prime powers}
Let $j$ be a singular modulus of discriminant $D=-3f^2$, i.e. a singular modulus relative to an order $\Ocal_j \subseteq \Q(\sqrt{-3})$ of conductor $f$. Assume that $f=p^n$ is a perfect prime power with $n$ a positive even natural number. 
\begin{itemize}
\item If $p \neq 3$ is odd then $p$ divides exactly $N_{\Q(j)/\Q}(j)$.
\item If $p=2$ then $2^4$ divides exactly $N_{\Q(j)/\Q}(j)$.
\end{itemize}
\end{thm}

\begin{proof}
The proof of Theorem \ref{prime powers} will rely on the fomulas proved by K. Lauter and B. Viray in \cite{Viray}. Following the same notation of their paper, set for $n$ positive even
$$d_1=-3, \ f_1=1, \ d_2=-3p^{2n},\ f_2=p^n$$
so that $j_2=j$ is a singular modulus of discriminant $d_2$ and $j_1=0$ is the only singular modulus of discriminant $d_1$. Then
$$\prod_{\substack{j_1, j_2 \\ \disc j_i=d_i}}(j_1 - j_2)=\pm N_{\Q(j)/\Q}(j)$$
since all the singular moduli of the same discriminant are conjugated. If now $w_i$ denotes the number of units in the order $\Ocal_{d_i}$ for $i=1,2$, then by our assumptions we have $w_1=6$ and $w_2=2$. By Theorem 1.1 in \cite{Viray} we get
\begin{equation} \label{equazione}
 \vert N_{\Q(j)/\Q}(j) \vert  ^{2/3} = \prod_{\substack{x^2 \leq 9p^{2n} \\ x^2 \equiv 9p^{2n} \ \text{mod} \ 4 }} F\left( \frac{9p^{2n}-x^2}{4} \right)
\end{equation}
where $F$ is a function that takes non-negative integers of the form $\frac{9p^{2n}-x^2}{4}$ to possibly fractional prime powers. The precise definition of the function $F(\cdot)$ is somewhat involved and not needed for the proof of the theorem. For  completeness of exposition, we decided however to incude it in the next paragraph. Our treatment follows closely the proof of \cite[Theorem 1.1]{Viray}, where the function $F(\cdot)$ is defined.

Let $L/\Q$ be the minimal finite field extension containing $\Q(\sqrt{-3})$ with the property that, for every rational prime $\ell$ and every singular modulus $\mathfrak{J}$ relative to the order $\Ocal_j$, there exist elliptic curves $E_0$ and $E_{\mathfrak{J}}$ defined over the ring of integers $\mathcal{O}_L$ such that $j(E_0)=0$, $j(E_{\mathfrak{J}})=\mathfrak{J}$ and which have good reduction at every prime $\mu \subseteq L$ above $\ell$. Such an extension can always be found by \cite{Serre-Tate}, Sections 5 and 6. For a fixed prime $\mu \subseteq L$ let $L_\mu^{\text{unr}}$ be the maximal unramified extension of the $\mu$-completion of $L$ and denote by $A\subseteq L_\mu^{\text{unr}}$ its ring of integers. Then for every $n\in \N$ and every isomorphism $f \in \operatorname{Iso}_{A/\mu^n} (E_0 \text{ mod } \mu^n, E_\mathfrak{J} \text{ mod } \mu^n)$ between the reduced elliptic curves mod $\mu^n$, there is a canonical isomorphism of rings 
\begin{align*}
i_{f}: \End_{A/\mu^n} (E_0 \text{ mod } \mu^n)& \xrightarrow{\sim} \End_{A/\mu^n} (E_\mathfrak{J} \text{ mod } \mu^n) \\
g &\mapsto f \circ g \circ f^{-1}
\end{align*}

which allows to write
\begin{equation*}
\begin{tikzcd}
\Z[\frac{1+\sqrt{-3}}{2}] \cong \End_A (E_0) \arrow[r, hook] & \End_{A/\mu^n} (E_0 \text{ mod } \mu^n) \arrow[r,"\sim"',"i_f"] & \End_{A/\mu^n} (E_\mathfrak{J} \text{ mod } \mu^n)                 \\
                  &             & \Ocal_j \cong \End_A(E_\mathfrak{J}) \arrow[u, hook]
\end{tikzcd}
\end{equation*}
where the non-labelled inclusions are induced by the reductions mod $\mu^n$. Let $R_f$ be the order generated by the image of $\Z[\frac{1+\sqrt{-3}}{2}]$ and $\mathcal{O}_j$ in $\End_{A/\mu^n} (E_\mathfrak{J} \text{ mod } \mu^n)$, and denote by $D_f$ its discriminant. It is possible to prove that the discriminant $D_f$ of the order $R_f$ is of the form
$$D_f=\left(\frac{9p^{2n}-x^2}{4}\right)^2$$
for some $x\in \Z$ with $x^2\leq 9p^{2n}$ and $x^2 \equiv 9p^{2n} \text{ mod } 4$. Then for every prime ideal $\mu \subseteq L$ and every integer $m$ of the form $\frac{9p^{2n}-x^2}{4}$ we define
$$N_{m,\mu}:=\frac{1}{3C} \sum_{\mathfrak{J}} \sum_{n\geq 1} \# \{f \in \operatorname{Iso}_{A/\mu^n} (E_0 \text{ mod } \mu^n, E_\mathfrak{J} \text{ mod } \mu^n): D_f=m^2 \}$$
where the first sum is taken over all singular moduli $\mathfrak{J}$ relative to the order $\mathcal{O}_j$ and $C\in \Z$ is such that $C=1$ if $x=0$ and $C=2$ otherwise. We finally define 
$$F(m):=\prod_{\mu \subseteq L} \mu^{N_{m,\mu}}$$
where the product is taken over all the prime ideals $\mu \subseteq L$. One can prove that, if $F(m)$ is non-trivial, there exists a unique rational prime $\ell$ such that $F(m)$ is supported only at prime ideals above $\ell$. Since the conditions defining $F(m)$ are Galois invariant, one can consider $F(m)$ to be a fractional power of the prime $\ell$.

Identity (\ref{equazione}) now shows that in order to understand the factorization of $ N_{\Q(j)/\Q}(j)$ one should study the function $F\left( \frac{9p^{2n}-x^2}{4} \right)$ for different values of $x$. We begin by studying the case $x=\pm 3p^n$, \textit{i.e.}~the factorization of $F(0)$. We denote by $v_p(\cdot)$ the usual $p$-adic valuation. Then by  the final part of Theorem 1.5 in \cite{Viray}, since $f_1=1$ and $d_2=d_1 p^{2n}$ we have 
$$v_p(F(0))=\frac{2}{6} \# \Pic(\Ocal_{d_1})=\frac{1}{3}$$
because $\Z[\frac{1+\sqrt{-3}}{2}]$ is a principal ideal domain. Combining this with equation (\ref{equazione}) gives
\begin{equation}\label{equazione2}
\vert N_{\Q(j)/\Q}(j) \vert  ^{2/3} = p^{2/3}\cdot \prod_{\substack{x^2 < 9p^{2n} \\ x^2 \equiv 9p^{2n} \ \text{mod} \ 4 }} F\left( \frac{9p^{2n}-x^2}{4} \right).
\end{equation}
In what follows we will distinguish between the cases $p$ odd and $p=2$. In the first case we will have to prove that none of the factors appearing in the product on the right-hand side of equation (\ref{equazione2}) is a power of $p$. In the second case we shall prove that there are exactly two factors in the same product that are equal to 2. \\
\textbf{Case 1: $p\neq 3$ odd.} We are now supposing that $x\neq \pm 3p^n$, \textit{i.e.}~that $m=\frac{9p^{2n}-x^2}{4}>0$. By the final part of Theorem 1.1 in \cite{Viray} we can have $v_p(F(m)) \neq 0$ only if $p$ divides $m$. Hence we only have to study the values of $F(m)$ with $p\mid m$. By definition of $m$ this implies that $p$ divides $x$ and we can then write $x=p^r k$, $0<r\leq n$ (here we use the fact that $p$ is odd), $k$ coprime with $p$. Hence $m$ can be factored as
$$m=\frac{9p^{2n}-k^2 p^{2r}}{4}=p^{2r} A, \hspace{1cm} A=\frac{9p^{2(n-r)}-k^2}{4}.$$
Notice that $p$ does not divide $A$. \\
By Theorem 1.5 in \cite{Viray} (which we can apply since $f_1=1$) we have that
\begin{equation}\label{equation3}
v_p(F(m))=\rho(m) \Agotic\left(\frac{m}{p^{1+n}} \right)
\end{equation}
where $\rho(\cdot)$ and $\Agotic(\cdot)$ are two functions defined for every integer $m,N$ as follows:
$$\rho(m) = \begin{cases}
 0 & \mbox{if } (-3,-m)_3=-1 \\ 
 1 & \mbox{if } 3\nmid m \\
 2 & \mbox{otherwise}
 \end{cases}$$
 
\begin{align*}
\Agotic(N)=\# \left\{ \agotic \subseteq \Z \left[\frac{1+\sqrt{-3}}{2}\right] \text{ideals}: N(\agotic)=N \right\}.
\end{align*}
where $(\cdot, \cdot)_3$ denotes the usual Hilbert symbol at 3. By Proposition 7.12 in \cite{Viray} we have that the right-hand side of \eqref{equation3} is zero if either $\frac{m}{p^{1+n}}$ is not an integer or $p \nmid d_1$ and $v_p(\frac{m}{p^{1+n}})\equiv 1 \mod 2$ (cf. the function $\varepsilon$ therein described). But now $p\neq 3$ by assumption and we have
$$v_p\left(\frac{m}{p^{1+n}} \right)=v_p(p^{2r-1-n}A)=2r-1-n \equiv 1 \mod 2$$
since by hypothesis $n$ is even. Hence the right-hand side of equation (\ref{equation3}) is 0 for every $m\neq 0$ and this concludes the proof in this case. \\
\textbf{Case 2: $p=2$.} As in the previous case, we have that $v_2(F(m))$ can be nonzero only if $2$ divides $m$, and this leads us to consider integers $m$ of the form
\begin{equation} \label{forma di m}
m=\frac{2^{2n} 9 - 2^{2r} k^2}{4} >0
\end{equation}
 where $k$ is either 0 or coprime with 2. First we study what happens for $k=0$. In this case we have $m=2^{2n-2}9$ and as above
$$v_2(F(m))=\rho(m) \Agotic\left(\frac{m}{2^{1+n}} \right)$$
where the quantity on the right-hand side is again zero if either $\frac{m}{2^{1+n}}$ is not an integer or $v_2(\frac{m}{2^{1+n}}) \equiv 1 \mod 2$. But we see that $\frac{m}{2^{1+n}}=2^{n-3}9$ and since $n$ is even by assumption, we deduce that $v_2(F(m))=0$ in this case. Hence we may assume $k \neq 0$ and coprime with 2. Notice that \eqref{forma di m} implies $r\leq n+1$. We consider two cases.
\begin{itemize}
\item[(i)] Suppose that $r\leq n$. In this case we can write
$$m=\frac{2^{2r}(2^{2(n-r)}9-k^2)}{4}=2^{2r-2}(2^{2(n-r)}9-k^2).$$
As in the previous case we have 
$$v_2(F(m))=\rho(m) \Agotic\left(\frac{m}{2^{1+n}} \right)$$
and we need to study 
$$\frac{m}{2^{1+n}}=2^{2r-n-3}(2^{2(n-r)}9-k^2).$$
Notice now that, since $k$ is coprime with $2$, the quantity inside the parenthesis cannot be divided by 2 unless $n=r$ and $k\in \{\pm 1\}$. Suppose first that $n \neq r$: then
$$v_2\left( \frac{m}{2^{1+n}} \right)=2r-n-3 \equiv 1 \mod 2$$
since $n$ is even by assumption. Using Theorem 7.12 in \cite{Viray} we deduce that $v_2(F(m))=0$ in this case. \\
Suppose now that $n=r$ and $k=\pm 1$: under these hypotheses we have $m=2^{2n+1}$ and
$$ v_2\left( \frac{m}{2^{1+n}} \right)=v_2(2^n)=n \equiv 0 \mod 2$$
by our assumptions on $n$. To compute the value of this valuation we have to use the full strength of Theorem 7.12 in \cite{Viray}: using the same notation of that theorem we have 
$$v_2(F(m))= \varepsilon_2(2^n) \prod_{\substack{q \mid 2^n\\ q\neq 2}} (*)$$
where we see that the product on the right is empty, hence equal to 1, and by definition of $\varepsilon_2(\cdot)$ we have $ \varepsilon_2(2^n)=1$. Hence for $k=\pm 1$, we have $v_2(F(m))=1$. 

\item[(ii)] Suppose now that $r=n+1$ and $k=\pm 1$. In this case $m=2^{2n-2} 5$ and we see that $v_2\left(\frac{m}{2^{n+1}} \right)=v_2(2^{n-3}5)=n-3$ is odd. As before we conclude that $v_2(F(m))=0$ in this case.
\end{itemize}
To sum up, if $p=2$ the only integers $m$ of the form $m=\frac{9p^{2n}-x^2}{4}$ for which $F(m)$ is a power of $2$ are $m=0$ ($x=\pm 2^n 3$) and $m=2^{2n+1}$ ($x=\pm 2^{2n}$), in which cases we obtain
$$F(0)=2^{1/3}, \hspace{1cm} F(2^{2n+1})=2.$$
Combining these results with equation (\ref{equazione}) we get
$$\vert N_{\Q(j)/\Q}(j) \vert ^{2/3}=2^{2/3} \cdot 2^2 \cdot B$$
where $B$ is an integer coprime with 2. This concludes the proof.
\end{proof}

The main problem encountered in trying to generalize Theorem \ref{prime powers} to the cases $n$ odd or $p=3$ is that, in these cases, the use of Lauter-Viray formulas requires some knowledge on the prime factorization of integers of the form $\frac{9p^{2(n-r)}-k^2}{4}$ and, in particular, on some congruence conditions modulo $3$ satisfied by these primes.

\section{THE CASE $j-1728$} 
In this final section we prove a result analogous to Theorem \ref{main theorem} for the difference $j-1728$, with $j$ a singular modulus. In this case we have the following result.

\begin{thm} \label{j-1728}
Let $S$ be the set of primes congruent to $1$ modulo $4$ and let $j$ be a singular modulus. If the difference $j-1728$ is an $S$-unit, then it is a unit.
\end{thm}

\begin{proof}
The argument is analogous to the one given in the proof of Theorem \ref{main theorem}, so we will omit the details. 

If the singular modulus $j$ is not relative to an order in $\Q(i)$, then an argument identical to the one given in the proof of Theorem \ref{main theorem} allows to conclude. For singular moduli corresponding to orders in $\Q(i)$ we have the following
\begin{claim} \label{claim2}
Let $j$ be a singular modulus relative to an order $\Ocal_j$ in $\Q(i)$. Then the primes $2,3$ and $7$ divide $N_{\Q(j)/\Q}(j-1728)$.
\end{claim}
The theorem follows from Claim \ref{claim2}. We prove the claim for $p=7$, the other cases being analogous.

Fix $j$ as in the statement of the claim and let $E/L$ be an elliptic curve defined over a number field $L$ with complex multiplication by the order $\Ocal_j$ and assume that $E$ has good reduction at every prime $\Pgotic$ of $L$ lying over 7. Fix such a prime: since 7 is inert in $\Q(i)$, by Theorem \ref{Lang} the reduced elliptic curve $\tilde{E}= E \mod \Pgotic$ is supersingular. However there is only one isomorphism class of supersingular elliptic curves over $\overline{\F}_7$, one representative being given by
$$E_0: y^2=x^3+x$$
with $j$-invariant $j(E_0)=1728$. Then
$$j \Mod \Pgotic=j(E_0)=1728.$$
We deduce that $\Pgotic$ divides $j-1728$ and this proves the claim.

\end{proof}

\begin{cor}\label{corollary 2}
Let $S$ be the set of rational primes congruent to 1 modulo 4. Then, for every singular modulus $j$, the difference $j-1728$ is not an $S$-unit.
\end{cor}

\begin{proof}
The result follows by combining Theorem \ref{j-1728} and Corollary 1.3 in \cite{Yingkun}.
\end{proof}

\newpage

\section*{APPENDIX: SOME NUMERICAL COMPUTATIONS}

In this appendix we collect in a table some numerical computations, obtained using SAGE (\cite{Sage}), concerning the norm factorizations for singular moduli of discriminant $-3f^2$. In the first column of the table we list the conductors $f$  of different orders of complex multiplication inside $\Q(\sqrt{-3})$; in the second column we compute, up to a sign, the norm factorizations of the corresponding singular moduli (since singular moduli relative to the same order form a Galois orbit in $\Qbar$, they all have the same norm). The factorizations are obtained simply by factoring the constant term in the Hilbert class polynomial of \samepage{discriminant $-3f^2$.} \\

\begin{center} 
\begin{longtable}{|C{0.05\textwidth}|C{0.90\textwidth}|}
\caption{Norm factorizations of singular moduli of discriminant $-3f^2$ for $f\in \{1,...,50\}$} \label{Tabella}\\
\hline
\textbf{f} & $|N_{\Q(j)/\Q }(j)|$ \Tstrut\Bstrut \\
\hline
\endfirsthead
\multicolumn{2}{C{0.24\textwidth}}%
{} \\
\hline
\textbf{f} & $|N_{\Q(j)/\Q }(j)|$  \Tstrut\Bstrut \\
\hline
\endhead
\hline \multicolumn{2}{r}{} \\
\endfoot
\hline
\captionsetup{belowskip=10pt}

\endlastfoot

1 & 0 \Tstrut \\
2& {$ 2^4 \cdot 3^3 \cdot 5^3 $ } \Tstrut \\
3 & {$2^{15} \cdot 3  \cdot 5^3$} \Tstrut \\
4 & {$2^4 \cdot 3^9 \cdot 5^6 \cdot 11^3$} \Tstrut \\
5 & {$2^{30} \cdot 3^6 \cdot 5 \cdot 11^3$} \Tstrut \\
6 & {$ 2^{12} \cdot 3^3 \cdot 5^9 \cdot 11^6 \cdot 17^3$} \Tstrut \\
7 & {$2^{30} \cdot 3^9 \cdot 5^6 \cdot 7 \cdot 17^3$} \Tstrut \\
8 & {$ 2^4 \cdot 3^{12} \cdot 5^{12} \cdot 11^6 \cdot 17^6 \cdot 23^3$} \Tstrut \\
9 & {$2^{45}\cdot 3 \cdot 5^9 \cdot 11^3 \cdot 23^3$} \Tstrut \\
10 & {$2^{24} \cdot 3^{30} \cdot 5^3 \cdot 11^6 \cdot 17^6\cdot 23^6 \cdot 29^3$} \Tstrut \Bstrut \\
11 & {$2^{63} \cdot 3^{12} \cdot 5^{12} \cdot 11 \cdot 17^3 \cdot 29^3$} \Tstrut  \\
12 & {$2^{12} \cdot 3^6 \cdot 5^{18} \cdot 11^6 \cdot 17^6 \cdot 23^6 \cdot 29^6$} \Tstrut \\
13 & {$2^{66} \cdot 3^{21} \cdot 5^{12} \cdot 11^6 \cdot 13 \cdot 23^3$} \Tstrut \\
14 & {$2^{24} \cdot 3^{18} \cdot 5^{18} \cdot 11^{12} \cdot 17^3 \cdot 23^6 \cdot 29^6 \cdot 41^3$} \Tstrut \\
15 & {$2^{96} \cdot 3^6 \cdot 5^3 \cdot 11^6 \cdot 17^6 \cdot 29^3 \cdot 41^3$} \Tstrut \\
16 & {$2^4 \cdot 3^{42} \cdot 5^{24} \cdot 11^6 \cdot 17^6 \cdot 23^3 \cdot 29^6 \cdot 41^6 \cdot 47^3$} \Tstrut \\
17 & {$2^{96} \cdot 3^{18} \cdot 5^{18} \cdot 11^6 \cdot 17 \cdot 23^6 \cdot 47^3$} \Tstrut \\
18 & {$ 2^{36}\cdot 3^3 \cdot 5^{27} \cdot 11^{12} \cdot 17^9 \cdot 23^6 \cdot 29^3 \cdot 41^6 \cdot 47^6 \cdot 53^3$} \Tstrut \\
19 & {$2^{93} \cdot 3^{27} \cdot 5^{18} \cdot 11^6 \cdot 19 \cdot 29^6 \cdot 41^3 \cdot 53^3$} \Tstrut \\
20 & {$2^{24} \cdot 3^{36} \cdot 5^6 \cdot 11^{15} \cdot 17^{12} \cdot 23^6 \cdot 29^6 \cdot 41^6 \cdot 47^6 \cdot 53^6 \cdot 59^3$} \Tstrut \\
21 & {$ 2^{96} \cdot 3^6 \cdot 5^{18} \cdot 11^{12} \cdot 17^3 \cdot 47^3 \cdot 59^3$} \Tstrut \\
22 & {$2^{48} \cdot 3^{60} \cdot 5^{36} \cdot 17^9 \cdot 23^6 \cdot 29^6 \cdot 41^3 \cdot 47^6 \cdot 53^6 \cdot 59^6$} \Tstrut \\ 
23 & {$2^{126} \cdot 3^{24} \cdot 5^{24} \cdot 11^9 \cdot 17^9 \cdot 23 \cdot 41^6 \cdot 53^3$} \Tstrut \\
24 & {$2^{12} \cdot 3^{12} \cdot 5^{36} \cdot 11^{12} \cdot 17^6 \cdot 23^{12} \cdot 29^6 \cdot 41^6 \cdot 47^3 \cdot 53^6 \cdot 59^6 \cdot 71^3$} \Tstrut \\
25 & {$2^{156} \cdot 3^{48} \cdot 5 \cdot 11^9 \cdot 17^6 \cdot 23^6 \cdot 47^6 \cdot 59^3 \cdot 71^3$} \Tstrut \\
26 & {$2^{48} \cdot  3^{36} \cdot 5^{36} \cdot 11^{12} \cdot 17^{12} \cdot 23^6 \cdot 29^9 \cdot 41^6 \cdot 47^6 \cdot 53^3 \cdot 59^6 \cdot 71^6$ } \Tstrut  \\
27 & {$2^{144} \cdot 3\cdot 5^{27} \cdot 11^{15} \cdot 17^9 \cdot 23^3 \cdot 29^6 \cdot 53^6$} \Tstrut \\
28 & {$2^{24} \cdot 3^{66} \cdot 5^{36} \cdot  11^{12} \cdot 17^6 \cdot 23^6 \cdot 41^6 \cdot 47^6 \cdot 53^6 \cdot 59^3 \cdot 71^6 \cdot 83^3 $} \Tstrut \\
29 & {$2^{150} \cdot 3^{30} \cdot 5^{30} \cdot 11^{12} \cdot 17^{12}\cdot  23^3 \cdot 29 \cdot 59^6  \cdot 71^3 \cdot 83^3$} \Tstrut \\
30 & {$2^{72}\cdot  3^{18}\cdot  5^9\cdot  11^{24} \cdot 17^{12}\cdot  23^{12} \cdot 29^9 \cdot 41^9 \cdot 47^6 \cdot 53^6 \cdot 59^6 \cdot 71^6 \cdot  83^6 \cdot 89^3$} \Tstrut \\
31 & {$2^{156} \cdot 3^{51} \cdot 5^{30} \cdot 11^{15} \cdot 17^9 \cdot 23^3 \cdot 29^3 \cdot 31 \cdot 41^6 \cdot 89^3$} \Tstrut \\
32 & {$2^4 \cdot 3^{48} \cdot 5^{48} \cdot 11^{24} \cdot 17^6 \cdot 23^{12} \cdot 29^{12} \cdot 47^9 \cdot 53^6 \cdot 59^6 \cdot 71^3 \cdot 83^6 \cdot 89^6$} \Tstrut\\
33 & {$2^{189} \cdot 3^{12} \cdot 5^{36} \cdot 17^9 \cdot 23^{12} \cdot 29^3 \cdot 47^6 \cdot 71^6 \cdot 83^3$} \Tstrut\\
34 & {$2^{72} \cdot 3^{96} \cdot 5^{54} \cdot 11^{24} \cdot 23^6 \cdot 29^6 \cdot 41^6 \cdot 53^9 \cdot 59^6 \cdot 71^6 \cdot 83^6 \cdot 89^6 \cdot 101^3$} \Tstrut\\
35 & {$ 2^{192} \cdot 3^{36} \cdot 5^6 \cdot 11^{12} \cdot 17^6 \cdot 23^{12} \cdot 29^6 \cdot 41^3 \cdot 53^6 \cdot 89^3 \cdot 101^3$} \Tstrut \\
36 & {$2^{36} \cdot 3^6 \cdot 5^{54} \cdot 11^{21} \cdot 17^{18} \cdot 23^6 \cdot 29^{12} \cdot 41^6 \cdot 47^6 \cdot 59^9 \cdot 71^6 \cdot 83^3 \cdot 89^6 \cdot 101^6 \cdot 107^3$} \Tstrut \\
37 & {$2^{186} \cdot 3^{63} \cdot 5^{36} \cdot 11^{15} \cdot 17^6 \cdot 23^6 \cdot 37 \cdot 47^3 \cdot 59^6 \cdot 83^6 \cdot 107^3$} \Tstrut \\
38 & {$2^{72} \cdot 3^{54} \cdot 5^{54} \cdot 11^{24} \cdot 17^{18} \cdot 23^{12} \cdot 29^6 \cdot 41^6 \cdot 47^6 \cdot 53^6 \cdot 71^6 \cdot 83^6 \cdot 89^3 \cdot 101^6 \cdot 107^6 \cdot 113^3$} \Tstrut \\
39 & {$2^{186} \cdot 3^{12} \cdot 5^{36} \cdot 11^{12} \cdot 17^{12} \cdot 23^6 \cdot 29^3 \cdot 41^6 \cdot 53^3 \cdot 89^6 \cdot 101^3 \cdot 113^3 $} \Tstrut \Bstrut \\
40 & {$ 2^{24} \cdot 3^{126} \cdot 5^{12} \cdot 11^{30} \cdot 17^{12} \cdot 23^6 \cdot 29^{12} \cdot 41^6 \cdot 47^6 \cdot 53^6 \cdot 59^6 \cdot 71^9 \cdot 83^6 \cdot 89^6 \cdot 101^6 \cdot 107^6 \cdot 113^6$} \\
41 & {$2^{228} \cdot 3^{42} \cdot 5^{42} \cdot 11^{12} \cdot 17^6 \cdot 23^9 \cdot 29^{12} \cdot 41 \cdot 47^6 \cdot 59^3 \cdot 71^6 \cdot 107^3$} \Tstrut \\
42 & {$2^{72} \cdot 3^{18} \cdot 5^{57} \cdot 11^{12} \cdot 17^{15} \cdot 23^{12} \cdot 29^{12} \cdot 41^3 \cdot 47^6 \cdot 53^6 \cdot 59^6 \cdot 83^6 \cdot 89^6 \cdot 101^3 \cdot 107^6 \cdot 113^6$} \Tstrut \\
43 & {$2^{222} \cdot 3^{75} \cdot 5^{45} \cdot 11^{12} \cdot 17^{12} \cdot 23^6 \cdot 29^9 \cdot 43 \cdot 53^6 \cdot 101^6 \cdot 113^3$} \Tstrut \\
44 & {$2^{48} \cdot 3^{72} \cdot 5^{78} \cdot 11^3 \cdot 17^{18} \cdot 23^{18} \cdot 29^6 \cdot 41^{18} \cdot 53^6 \cdot 59^6 \cdot 71^6 \cdot 83^9 \cdot 89^6 \cdot 101^6 \cdot 107^3 \cdot 113^6 \cdot 131^3 $}\Tstrut \\
45 & {$2^{294} \cdot 3^6 \cdot 5^9 \cdot 11^{21} \cdot 17^{18} \cdot 23^{12} \cdot 29^9 \cdot 41^3 \cdot 59^6 \cdot 71^3 \cdot 83^6 \cdot 107^6 \cdot 131^3$} \Tstrut \Bstrut \\
46 & {$2^{96} \cdot 3^{126} \cdot 5^{78} \cdot 11^{30} \cdot 17^{15} \cdot 29^6 \cdot 41^6 \cdot 47^{12} \cdot 59^6 \cdot 71^6 \cdot 89^9 \cdot 101^6 \cdot 107^6 \cdot 113^3 \cdot 131^6 \cdot 137^3 $}  \\
47 & {$2^{258} \cdot 3^{48} \cdot 5^{51} \cdot 11^{24} \cdot 17^{12} \cdot 23^9 \cdot 29^9 \cdot 41^3 \cdot 47 \cdot 89^6 \cdot 113^6 \cdot 137^3$} \Tstrut \Bstrut \\
48 & {$ 2^{12} \cdot 3^{24} \cdot 5^{78} \cdot 11^{36} \cdot 17^{18} \cdot 23^6 \cdot 29^6 \cdot 41^{12} \cdot 47^9 \cdot 53^{12} \cdot 71^6 \cdot 83^6 \cdot 101^6 \cdot 107^6 \cdot 113^6 \cdot 131^6 \cdot 137^6 $}  \\
49 & {$2^{222} \cdot 3^{69} \cdot 5^{42} \cdot 7 \cdot 11^{18} \cdot 17^{12} \cdot 23^6 \cdot 29^6 \cdot 47^3 \cdot 71^6 \cdot 83^3 \cdot 131^3$} \Tstrut \Bstrut \\
50 & {$2^{120} \cdot 3^{90} \cdot 5^3 \cdot 11^{42} \cdot 17^{24} \cdot 23^{12} \cdot 29^{15} \cdot 41^{12} \cdot 47^{12} \cdot 53^6 \cdot 59^{12} \cdot 71^6 \cdot 83^6\cdot 89^6 \cdot 101^9 \cdot 107^6 \cdot 113^6 \cdot 131^6 \cdot 137^6 \cdot 149^3$} \Bstrut \\

\end{longtable}
\end{center}

\section*{Acknowledgments}
The author would like to thank his supervisor Fabien Pazuki, for his guidance and advice, and Philipp Habegger for the helpful suggestions and comments. He would also like to thank Riccardo Pengo and Peter Stevenhagen for the useful discussions, and the anonymous referee for the careful reading and the many insightful comments. 

\vspace{\baselineskip}
\noindent
\framebox[\textwidth]{
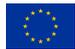
\begin{tabular*}{0.96\textwidth}{@{\extracolsep{\fill} }cp{0.84\textwidth}}
\raisebox{-0.7\height}{%
    \begin{tikzpicture}[y=0.80pt, x=0.8pt, yscale=-1, inner sep=0pt, outer sep=0pt, 
    scale=0.12]
    \definecolor{c003399}{RGB}{0,51,153}
    \definecolor{cffcc00}{RGB}{255,204,0}
    \begin{scope}[shift={(0,-872.36218)}]
      \path[shift={(0,872.36218)},fill=c003399,nonzero rule] (0.0000,0.0000) rectangle (270.0000,180.0000);
      \foreach \myshift in 
           {(0,812.36218), (0,932.36218), 
    		(60.0,872.36218), (-60.0,872.36218), 
    		(30.0,820.36218), (-30.0,820.36218),
    		(30.0,924.36218), (-30.0,924.36218),
    		(-52.0,842.36218), (52.0,842.36218), 
    		(52.0,902.36218), (-52.0,902.36218)}
        \path[shift=\myshift,fill=cffcc00,nonzero rule] (135.0000,80.0000) -- (137.2453,86.9096) -- (144.5106,86.9098) -- (138.6330,91.1804) -- (140.8778,98.0902) -- (135.0000,93.8200) -- (129.1222,98.0902) -- (131.3670,91.1804) -- (125.4894,86.9098) -- (132.7547,86.9096) -- cycle;
    \end{scope}
    \end{tikzpicture}%
}
&
This project has received funding from the European Union Horizon 2020 research and
innovation programme under the Marie Sk{\l}odowska-Curie grant agreement No 801199.
\end{tabular*}
}
  

\bibliographystyle{amsplain}

\end{document}